\documentclass[12pt, leqno]{amsart}

\usepackage{amssymb, hyperref}

\usepackage[all]{xy}

\theoremstyle{plain}
\newtheorem{thm}{Theorem}
\newtheorem{theorem}{Theorem}[section]
\newtheorem{lemma}[theorem]{Lemma}

\theoremstyle{definition}

\newtheorem{remark}[theorem]{Remark}

\newcommand\bC{{\mathbb C}}

\newcommand\bG{{\mathbb G}}

\newcommand\bP{{\mathbb P}}
\newcommand\bQ{{\mathbb Q}}
\newcommand\bR{{\mathbb R}}

\newcommand\bZ{{\mathbb Z}}

\newcommand\cO{{\mathcal O}}

\newcommand\fg{\mathfrak{g}}
\newcommand\fh{\mathfrak{h}}
\newcommand\fm{\mathfrak{m}}

\newcommand\rM{{\rm M}}

\newcommand\wA{\widehat{A}}

\newcommand\aff{{\rm aff}}

\newcommand\ant{{\rm ant}}

\newcommand\charac{{\rm char}}

\newcommand\gp{{\rm gp}}

\newcommand\id{{\rm id}}

\newcommand\num{{\rm num}}

\newcommand\red{{\rm red}}

\newcommand\Alb{{\rm Alb}}
\newcommand\Aut{{\rm Aut}}

\newcommand\Chow{{\rm Chow}}

\newcommand\End{{\rm End}}

\newcommand\GL{{\rm GL}}
\newcommand\Grass{{\rm Grass}}
\newcommand\Hilb{{\rm Hilb}}

\newcommand\Int{{\rm Int}}
\newcommand\Ker{{\rm Ker}}
\newcommand\Lie{{\rm Lie}}

\newcommand\N{{\rm N}}

\newcommand\NS{{\rm NS}}
\newcommand\Pic{{\rm Pic}}
\newcommand\PGL{{\rm PGL}}

\newcommand\Spec{{\rm Spec}}

\newcommand\Supp{{\rm Supp}}

\numberwithin{equation}{section}

\title{Automorphism groups of almost homogeneous varieties}

\author{Michel Brion}

\date{}

\begin{document}

\begin{abstract}
Consider a smooth connected algebraic group $G$ acting on
a normal projective variety $X$ with an open dense orbit. 
We show that $\Aut(X)$ is a linear algebraic group if so is $G$; 
for an arbitrary $G$, the group of components of $\Aut(X)$ is 
arithmetic. Along the way, we obtain a restrictive condition 
for $G$ to be the full automorphism group of some normal 
projective variety. 
\end{abstract}

\maketitle

\section{Introduction}
\label{sec:int}

Let $X$ be a projective algebraic variety over an algebraically 
closed field $k$. It is known that the automorphism group 
$\Aut(X)$ has a natural structure of smooth $k$-group scheme, 
locally of finite type (see \cite[p.~268]{Gro61}. 
This yields an exact sequence 
\[ 1 \longrightarrow \Aut^0(X) \longrightarrow \Aut(X)
\longrightarrow \pi_0 \, \Aut(X) \longrightarrow 1, \]
where $\Aut^0(X)$ is (the group of $k$-rational
points of) a smooth connected algebraic group, and 
$\pi_0 \, \Aut(X)$ is a discrete group. 

To analyze the structure of $\Aut(X)$, one may
start by considering the connected automorphism
group $\Aut^0(X)$ and the group of components 
$\pi_0 \, \Aut(X)$ separately. It turns out that
there is no restriction on the former: every smooth
connected algebraic group is the connected 
automorphism group of some normal projective variety
$X$ (see \cite[Thm.~1]{Br14}). In characteristic $0$, 
we may further take $X$ to be smooth by using 
equivariant resolution of singularities (see 
e.g.~\cite[Chap.~3]{Kol07}).

By constrast, little is known on the structure 
of the group of components. Every finite group $G$
can be obtained in this way, as $G$ is the full automorphism 
group of some smooth projective curve (see the main
result of \cite{MV}). But the group of components is 
generally infinite, and it is unclear how infinite it can be.

The long-standing question whether this group
is finitely generated has been recently answered in the
negative by Lesieutre. He constructed an example 
of a smooth projective variety of dimension $6$ having 
a discrete, non-finitely generated automorphism group 
(see \cite{Lesieutre}). His construction has been 
extended in all dimensions at least $2$ by Dinh and Oguiso, 
see \cite{DO}. The former result is obtained over an 
arbitrary field of characteristic $0$, while the latter 
holds over the complex numbers; it is further extended to 
odd characteristics in \cite{Oguiso}. On the positive side, 
$\pi_0 \, \Aut(X)$ is known to be finitely presented for 
some interesting classes of projective varieties, including 
abelian varieties (see \cite{Bor62}) and complex 
hyperk\"ahler manifolds (see \cite[Thm.~1.5]{CF}).

In this article, we obtain three results on automorphism
groups, which generalize recent work. The first one goes 
in the positive direction for  \emph{almost homogeneous} 
varieties, i.e., those on which a smooth connected algebraic 
group acts with an open dense orbit.

\begin{thm}\label{thm:ahv}
Let $X$ be a normal projective variety, almost homogeneous 
under a linear algebraic group. Then $\Aut(X)$ is a linear 
algebraic group as well.
\end{thm}

This was first obtained by Fu and Zhang in the setting of
compact K\"ahler manifolds (see \cite[Thm.~1.2]{FZ}).
The main point of their proof is to show that the 
anticanonical line bundle is big. This relies on Lie-theoretical
methods, in particular the $\mathfrak{g}$-anticanonical fibration 
of \cite[I.2.7]{HO}, also known as the Tits fibration. But this
approach does not extend to positive characteristics, already 
when $X$ is homogeneous under a semi-simple algebraic group: 
then any big line bundle on $X$ is ample, but $X$ is generally 
not Fano (see \cite{HL}). 

To prove Theorem \ref{thm:ahv}, we construct a normal projective
variety $X'$ equipped with a birational morphism $f: X' \to X$ 
such that the action of $\Aut(X)$ on $X$ lifts to an action on 
$X'$ that fixes the isomorphism class of a big line bundle. 
For this, we use a characteristic-free version of the Tits fibration
(Lemma \ref{lem:tits}). 

Our second main result goes in the negative direction, as it
yields many examples of algebraic groups which cannot
be obtained as the automorphism group of a normal
projective variety. To state it, we introduce some notation.

Let $G$ be a smooth connected algebraic group. 
By Chevalley's structure theorem (see \cite{Conrad} 
for a modern proof), there is a unique exact sequence 
of algebraic groups
\[ 1 \longrightarrow G_{\aff} \longrightarrow G 
\longrightarrow A \longrightarrow 1, \]
where $G_{\aff}$ is a smooth connected affine (or 
equivalently linear) algebraic group and $A$ is an abelian 
variety. We denote by $\Aut^{G_{\aff}}_{\gp}(G)$ the group 
of automorphisms of the algebraic group $G$ which fix 
$G_{\aff}$ pointwise.

\begin{thm}\label{thm:inf}
With the above notation, assume that the group
$\Aut^{G_{\aff}}_{\gp}(G)$ is infinite. If $G \subset \Aut(X)$ 
for some normal projective variety $X$, then $G$ has infinite 
index in $\Aut(X)$. 
\end{thm}

It is easy to show that $\Aut^{G_{\aff}}_{\gp}(G)$ 
is an arithmetic group, and to construct classes of examples 
for which this group is infinite, see Remark \ref{rem:inf}.  

If $G$ is an abelian variety, then $\Aut^{G_{\aff}}_{\gp}(G)$ 
is just its group of automorphisms as an algebraic group. 
In this case, Theorem \ref{thm:inf} is due (in essence) 
to Lombardo and Maffei, see \cite[Thm.~2.1]{LM}. They also 
obtain a converse over the field of complex numbers: 
given an abelian variety $G$ with finite automorphism group, 
they construct a smooth projective variety $X$ such that 
$\Aut(X) = G$ (see \cite[Thm.~3.9]{LM}). 

Like that of \cite[Thm.~2.1]{LM}, the proof of Theorem 
\ref{thm:inf} is based on the existence of a homogeneous 
fibration of $X$ over an abelian variety, the quotient
of $A$ by a finite subgroup scheme. This allows us to
construct an action on $X$ of a subgroup of finite index 
of $\Aut^{G_{\aff}}_{\gp}(G)$, which normalizes $G$ 
and intersects this group trivially. 

When $X$ is almost homogeneous under $G$, the Albanese
morphism provides such a homogeneous fibration,
as follows from \cite[Thm.~3]{Br10}. A finer analysis of its
automorphisms leads to our third main result.

\begin{thm}\label{thm:pos}
Let $X$ be a normal projective variety, almost homogeneous 
under a smooth connected algebraic group $G$. Then
$\pi_0 \, \Aut(X)$ is an arithmetic group. In positive 
characteristics, $\pi_0 \, \Aut(X)$ is commensurable with 
$\Aut_{\gp}^{G_{\aff}}(G)$.
\end{thm}

\noindent
(The second assertion does not hold in characteristic $0$, 
see Remark \ref{rem:pos}). 

These results leave open the question 
whether every \emph{linear} algebraic group is the automorphism 
group of a normal projective variety. Further open questions are 
discussed in the recent survey \cite{Cantat}, in the setting of smooth 
complex projective varieties.

\medskip

\noindent
{\bf Acknowledgments.} The above results have first been presented
in a lecture at the School and Workshop on Varieties and Group Actions
(Warsaw, September 23--29, 2018), with a more detailed and 
self-contained version of this article serving as lecture notes
(see \cite{Br19}). I warmly thank the organizers of this event 
for their invitation, and the participants for stimulating questions. 
Also, I thank Roman Avdeev, Yves de Cornulier, Fu Baohua, 
H\'el\`ene Esnault and Bruno Laurent for helpful discussions or email 
exchanges on the topics of this article, and the two anonymous referees
for their valuable remarks and comments. Special thanks are due 
to Serge Cantat for very enlightening suggestions and corrections, 
and to Ga\"el R\'emond for his decisive help with the proof of 
Theorem \ref{thm:pos}. 

This work was partially supported by the grant 346300 for IMPAN 
from the Simons Foundation and the matching 2015-2019 Polish 
MNiSW fund.

\section{Some preliminary results}
\label{sec:prel}

We first set some notation and conventions, which will be
valid throughout this article. We fix an algebraically closed 
ground field $k$ of characteristic $p \geq 0$. 
By a \emph{scheme}, we mean a separated scheme over $k$, 
unless otherwise stated; a \emph{subscheme} is a locally closed 
$k$-subscheme. Morphisms and products of schemes are 
understood to be over $k$ as well. A \emph{variety} is an integral 
scheme of finite type. An \emph{algebraic group} is a group scheme 
of finite type; a \emph{locally algebraic group} is a group scheme, 
locally of finite type. 

Next, we present some general results on automorphism groups, 
refering to \cite[Sec.~2]{Br19} for additional background 
and details. We begin with a useful observation:

\begin{lemma}\label{lem:closed}
Let $f: X \to Y$ be a birational morphism, where $X$ and $Y$ 
are normal projective varieties. Assume that the action of 
$\Aut(Y)$ on $Y$ lifts to an action on $X$. Then the 
corresponding homomorphism $\rho : \Aut(Y) \to \Aut(X)$ 
is a closed immersion.
\end{lemma}

\begin{proof}
Since $f$ restricts to an isomorphism on dense open
subvarieties of $X$ and $Y$, the scheme-theoretic
kernel of $\rho$ is trivial. Thus, $\rho$ induces
a closed immersion $\rho^0 : \Aut^0(Y) \to \Aut^0(X)$.
On the other hand, we have $f_*(\cO_X) = \cO_Y$
by Zariski's Main Theorem; thus, Blanchard's lemma
(see \cite[Prop.~4.2.1]{BSU}) yields a
homomorphism $f_* : \Aut^0(X) \to \Aut^0(Y)$.
Clearly, $\rho^0$ and  $f_*$ are mutually inverse; 
thus, the image of $\rho$ contains $\Aut^0(X)$.
Since $\Aut(X)/\Aut^0(X)$ is discrete, this 
yields the statement.
\end{proof}

We now discuss the action of automorphisms on line bundles. 
Given a projective variety $X$, the Picard group $\Pic(X)$ 
has a canonical structure of locally algebraic group (see 
\cite{Gro62}); its group of components, the N\'eron-Severi 
group $\NS(X)$, is finitely generated by \cite[XIII.5.1]{SGA6}. 
The action of $\Aut(X)$ on $\Pic(X)$ via pullback extends 
to an action of the corresponding group functor, and hence 
of the corresponding locally algebraic group. 
As a consequence, for any line bundle $\pi : L \to X$ 
with class $[L] \in \Pic(X)$, the reduced stabilizer 
$\Aut(X,[L])$ is closed in $\Aut(X)$. 

Given $L$ as above, the polarization map
\[ \Aut(X) \longrightarrow \Pic(X), \quad
g \longmapsto [g^*(L) \otimes L^{-1}] \]
takes $\Aut^0(X)$ to $\Pic^0(X)$. Therefore,
$\Aut^0(X)$ acts trivially on the quotient
$\Pic(X)/\Pic^0(X) = \NS(X)$. 
This yields an action of $\pi_0 \, \Aut(X)$ on $\NS(X)$ 
and in turn, on the quotient of $\NS(X)$ by its torsion subgroup: 
the group of line bundles up to numerical equivalence, 
that we denote by $\N^1(X)$. Also, we denote by $[L]_{\num}$
the class of $L$ in $\N^1(X)$, and by $\Aut(X,[L]_{\num})$
its reduced stabilizer in $\Aut(X)$. Then 
$\Aut(X,[L]_{\num})$ contains $\Aut^0(X)$, and hence
is a closed subgroup of $\Aut(X)$, containing $\Aut(X,[L])$.

Further, recall that we have a central extension 
of locally algebraic groups
\[ 1 \longrightarrow \bG_m \longrightarrow \Aut^{\bG_m}(L)
\longrightarrow \Aut(X,[L]) \longrightarrow 1, 
\]
where $\Aut^{\bG_m}(L)$ denotes the group of 
automorphisms of the variety $L$ which commute with
the $\bG_m$-action by multiplication on the fibers of
$\pi$. For any integer $n$, the space $H^0(X,L^{\otimes n})$
is equipped with a linear representation of $\Aut^{\bG_m}(L)$,
and hence with a projective representation of $\Aut(X,[L])$.
Moreover, the natural rational map 
\[ f_n : X \dasharrow \bP \, H^0(X,L^{\otimes n}) \]
(where the right-hand side denotes the projective space 
of hyperplanes in $H^0(X,L^{\otimes n})$) is equivariant
relative to the action of $\Aut(X,[L])$. 

Recall that $L$ is \emph{big} if $f_n$ is birational onto 
its image for some $n \geq 1$. (See \cite[Lem.~2.60]{KM} 
for further characterizations of big line bundles).

\begin{lemma}\label{lem:bignef}
Let $L$ a big and nef line bundle on a normal projective variety 
$X$. Then $\Aut(X,[L]_{\num})$ is an algebraic group.
\end{lemma}

\begin{proof}
It suffices to show that the locally algebraic group 
$G := \Aut(X,[L]_{\num})$ has finitely many 
components. For this, we adapt the arguments of 
\cite[Prop.~2.2]{Lieberman} and \cite[Lem.~2.23]{Zhang}. 

By Kodaira's lemma, we have $L^{\otimes n} \simeq A \otimes E$ 
for some positive integer $n \geq 1$, some ample line bundle $A$ 
and some effective line bundle $E$ on $X$ (see 
\cite[Lem.~2.60]{KM}). Since $G$ is a closed subgroup 
of $\Aut(X,[L^{\otimes n}]_{\num})$, we may assume that $n = 1$.

Consider the ample line bundle $A \boxtimes A$ on $X \times X$. 
We claim that the degrees of the graphs 
$\Gamma_g \subset X \times X$, where $g \in G$, are bounded 
independently of $g$. This implies the statement as follows:
the above graphs form a flat family of normal 
subvarieties of $X \times X$, parameterized by $G$ (a disjoint 
union of open and closed smooth varieties). This yields a morphism
from $G$ to the Hilbert scheme $\Hilb_{X \times X}$. Since
$G$ is closed in $\Aut(X)$ and the latter is the reduced subscheme
of an open subscheme of $\Hilb_{X \times X}$, this morphism 
is an immersion, say $i$. By \cite[I.6.3, I.6.6.1]{Kol99}, we may
compose $i$ with the Hilbert-Chow morphism to obtain a
local immersion $\gamma : G \to \Chow_{X \times X}$.
Clearly, $\gamma$ is injective, and hence an immersion.
Thus, the graphs $\Gamma_g$ are the $k$-rational points 
of a locally closed subvariety of $\Chow_{X \times X}$. 
Since the cycles of any prescribed degree form a subscheme 
of finite type of $\Chow_{X \times X}$, our claim yields that $G$ 
has finitely many components indeed.

We now prove this claim.
Let $d := \dim(X)$ and denote by $L_1 \cdots L_d$ the 
intersection number of the line bundles $L_1,\ldots,L_d$ 
on $X$; also, we denote the line bundles additively. 
By \cite[Thm.~9.6.3]{FGA}, $L_1 \cdots L_d$ only depends on
the numerical equivalence classes of $L_1,\ldots,L_d$.
With this notation, the degree of $\Gamma_g$ relative to 
$A \boxtimes A$ is the self-intersection number 
$(A + g^*A)^d$. We have
\[ (A + g^* A)^d = (A + g^* A)^{d-1} \cdot (L + g^* L)
- (A +g^* A)^{d-1} \cdot (E + g^* E). \]
We now use the fact that $L_1 \cdots L_{d-1} \cdot E \geq 0$
for any ample line bundles $L_1,\ldots,L_{d-1}$  
(this is a very special case of \cite[Ex.~12.1.7]{Fulton}), 
and hence for any nef line bundles $L_1,\ldots,L_{d-1}$ 
(since the nef cone is the closure of the ample cone).
It follows that 
\[ (A + g^* A)^d \leq (A + g^* A)^{d-1} \cdot (L + g^* L) \]
\[ = (A + g^* A)^{d-2} \cdot (L + g^* L)^2 - 
(A + g^* A)^{d-2} \cdot (L + g^* L) \cdot (E + g^* E). \]
Using again the above fact, this yields
\[ (A + g^* A)^d \leq (A + g^* A)^{d-2} \cdot (L + g^* L)^2. \]
Proceeding inductively, we obtain
$(A + g^* A)^d \leq (L+ g^* L)^d$.
Since $g^* L$ is numerically equivalent to $L$, this yields
the desired bound
\[ (A + g^* A)^d \leq 2^d \, L^d. \]
\end{proof}

\begin{lemma}\label{lem:big}
Let $L$ be a big line bundle on a normal projective variety 
$X$. Then $\Aut(X,[L])$ is a linear algebraic group.
\end{lemma}

\begin{proof}
Since $G := \Aut(X,[L])$ is a closed subgroup of 
$\Aut(X,[L^{\otimes n}])$ for any $n \geq 1$, 
we may assume that the rational map  
\[ f_1: X \dasharrow \bP \, H^0(X,L) \] 
is birational onto its closed image $Y_1$. Note that
$f_1$ is $G$-equivariant and the action of $G$ on 
$\bP \, H^0(X,L)$ stabilizes $Y_1$.

Consider the blowing-up of the base locus of $L$, as in 
\cite[Ex.~II.7.17.3]{Hartshorne}, and its normalization
$\tilde{X}$. Denote by
\[ \pi : \tilde{X} \longrightarrow X \]
the resulting birational morphism; then
$\pi_*(\cO_{\tilde{X}}) = \cO_X$ by Zariski's Main Theorem. 
Let $\tilde{L} := \pi^*(L)$; then 
$H^0(\tilde{X},\tilde{L}) \simeq H^0(X,L)$ and
$\tilde{L} = L' \otimes E$, where $L'$ is a line bundle generated 
by its subspace of global sections $H^0(X,L)$,
and $E$ is an effective line bundle. Thus, $L'$ is big.
Moreover, the action of $G$ on $X$ lifts to 
an action on $\tilde{X}$ which fixes both classes $[\tilde{L}]$ 
and $[L']$. We now claim that the image of the resulting
homomorphism $\rho : G \to \Aut(\tilde{X},[L'])$ is closed.

Note that $\rho$ factors through a homomorphism
\[ \eta : G \to \Aut(\tilde{X}, [\tilde{L}]) 
= \Aut(\tilde{X},[\tilde{L}], [L']), \]
where the right-hand side 
is a closed subgroup of $\Aut(\tilde{X},[L'])$. Thus, it
suffices to show that the image of $\eta$ is closed.
For this, we adapt the argument of Lemma \ref{lem:closed}.
Consider the cartesian diagram
\[ \xymatrix{
\tilde{L} \ar[r]^{\varphi} \ar[d]^{\tilde{\pi}} 
& L \ar[d]^{\pi}  \\  \tilde{X}  \ar[r]^{f} &  X.  \\
} \]
Since $f_*(\cO_{\tilde{X}}) = \cO_X$, we have
$\varphi_*(\cO_{\tilde{L}}) = \cO_L$. 
So Blanchard's lemma (see \cite[Prop.~4.2.1]{BSU})
yields a homomorphism 
$\varphi_* : \Aut^{\bG_m}(\tilde{L})^0 \to \Aut^{\bG_m}(L)^0$,
and hence a homomorphism
\[ f_* : \Aut(\tilde{X},\tilde{L})^0 \to \Aut(X,[L])^0 = G^0 \]
which is the inverse of 
$\eta^0 : G^0 \to \Aut(\tilde{X},\tilde{L})^0$. Thus,
the image of $\eta$ contains $\Aut(\tilde{X},\tilde{L})^0$.
This implies our claim.

By this claim, we may replace the pair ($X,L$) with 
($\tilde{X},L'$); equivalently, we may assume that 
the big line bundle $L$ is generated by its global sections. 
Then $\Aut(X,[L]_{\num})$ is an algebraic group by Lemma 
\ref{lem:bignef}. Since $G$ is a closed subgroup of 
$\Aut(X,[L]_{\num})$, it is algebraic as well. So the image 
of the homomorphism $G \to \Aut(\bP \, H^0(X,L)) \simeq \PGL_N$ 
is closed. As $f_1$ is birational, the scheme-theoretic kernel of 
this homomorphism is trivial; thus, $G$ is linear.

\end{proof}

Finally, we record a classical bigness criterion, for which we
could locate no reference in the generality that we need:

\begin{lemma}\label{lem:affbig}
Consider an effective Cartier divisor $D$ on a projective
variety $X$ and let $U := X \setminus \Supp(D)$. If $U$ 
is affine, then $\cO_X(D)$ is big.
\end{lemma}

\begin{proof}
Denote by $s \in H^0(X,\cO_X(D))$ the canonical section,
so that $U = X_s$ (the complement of the zero locus of $s$).
Let $h_1,\ldots,h_r$ be generators of the $k$-algebra
$\cO(U)$. Then there exist positive integers $n_1,\ldots,n_r$ 
such that $h_i \, s^{n_i}  \in H^0(X, \cO_X(n_i D))$ for 
$i = 1,\ldots,r$. So $h_i \,s^n \in H^0(X,\cO_X(nD))$ 
for any $n \geq n_1,\ldots,n_r$.
It follows that the rational map 
$f_n : X \dasharrow \bP \, H^0(X,\cO_X(nD))$ 
restricts to a closed immersion
$U \to \bP \, H^0(X,\cO_X(nD))_{s^n}$,
and hence is birational onto its image.
\end{proof}

\section{Proof of Theorem \ref{thm:ahv}}
\label{sec:ahv}

We first obtain a characteristic-free analogue of
the Tits fibration:

\begin{lemma}\label{lem:tits}
Let $G$ be a connected linear algebraic group, and 
$H$ a subgroup scheme. For any $n \geq 1$, denote 
by $G_n$ (resp.~$H_n$) the $n$-th infinitesimal neighborhood 
of the neutral element $e$ in $G$ (resp.~$H$).

\begin{enumerate}

\item[{\rm (i)}] The union of the $G_n$ ($n \geq 1$)
is dense in $G$.

\item[{\rm (ii)}] For $n \gg 0$, we have the equality 
$N_G(H^0) = N_G(H_n)$ of
scheme-theoretic normalizers.

\item[{\rm (iii)}] The canonical morphism
$f : G/H \to G/N_G(H^0)$ is affine.

\end{enumerate}

\end{lemma}

\begin{proof}
(i) Denote by $\fm \subset \cO(G)$ the maximal ideal
of $e$; then 
\[ G_n = \Spec(\cO(G)/\fm^n) \] 
for all $n$. Thus, the assertion is equivalent to 
$\bigcap_{n \geq 1} \fm^n = 0$.
This is proved in \cite[I.7.17]{Jantzen}; we recall 
the argument for the reader's convenience.
If $G$ is smooth, then $\cO(G)$ is a noetherian domain,
hence the assertion follows from Nakayama's lemma.
For an arbitrary $G$, we have an isomorphism of algebras  
\[ \cO(G) \simeq  \cO(G_{\red}) \otimes A, \] 
where $A$ is a local $k$-algebra of finite dimension
as a $k$-vector space (see \cite[III.3.6.4]{DG}). Thus, 
$\fm = \fm_1 \otimes 1 + 1 \otimes \fm_2$,
where $\fm_1$ (resp.~$\fm_2$) denotes the maximal ideal 
of $e$ in $\cO(G_{\red})$ (resp.~the maximal ideal of
$A$). We may choose an integer $N \geq 1$
such that $\fm_2^N = 0$; then 
$\fm^n \subset \fm_1^{n - N} \otimes A$ for all $n \geq N$.
Since $G_{\red}$ is smooth and connected, we have
$\bigcap_{n \geq 1} \fm_1^n = 0$ by the above step;
this yields the assertion.

(ii) Since $H_n = H^0 \cap G_n$ and $G$ normalizes $G_n$, 
we have $N_G(H^0) \subset N_G(H_n)$. 
To show the opposite inclusion, note that 
$(H_n)_{n-1} = H_{n-1}$, hence we have
$N_G(H_n) \subset N_G(H_{n-1})$. This decreasing
sequence of closed subschemes of $G$ stops, say
at $n_0$. Then $N_G(H_{n_0})$ normalizes 
$H_n$ for all $n \geq n_0$. In view of (i),
it follows that $N_G(H_{n_0})$ normalizes $H^0$.

(iii) We have a commutative triangle
\[ \xymatrix{
G/H^0 \ar[rd]^-{\psi} \ar[d]_{\varphi} \\
G/H \ar[r]^-{f} & G/N_G(H^0), \\
} \]
where $\varphi$ is a torsor under $H/H^0$ (a finite
constant group), and $\psi$ is a torsor under $N_G(H^0)/H^0$
(a linear algebraic group). In particular, $\psi$ is affine.
Let $U$ be an open affine subscheme of $G/N_G(H^0)$. 
Then $\psi^{-1}(U) \subset G/H^0$ is open, affine and
stable under $H/H^0$. Hence 
$f^{-1}(U) = \varphi (\psi^{-1}(U))$ is affine.
\end{proof}

\begin{remark}\label{rem:tits}
(i) The first infinitesimal neighborhood $G_1$
may be identified with the Lie algebra $\fg$
of $G$; thus, $G_1 \cap H^0 = H_1$
is identified with the Lie algebra $\fh$ of $H$. 

If $\charac(k) = 0$, then $N_G(H^0) = N_G(\fh)$,
since every subgroup scheme of $G$ is uniquely
determined by its Lie subalgebra (see 
e.g.~\cite[II.6.2.1]{DG}). As a consequence, the morphism
$f: G/H \to G/N_G(H^0)$ is the $\fg$-anticanonical 
fibration considered in \cite[I.2.7]{HO} (see also
\cite[\S 4]{Haboush}).

By contrast, if $\charac(k) > 0$ then the natural morphism 
$G/H \to G/N_G(\mathfrak{h})$ is not necessarily affine 
(see e.g.~\cite[Ex.~5.6]{Br10}). In particular, the inclusion 
$N_G(H^0) \subset N_G(\fh)$ may be strict.

\smallskip \noindent
(ii) If $\charac(k) = p > 0$, then $G_{p^n}$ is the
$n$th Frobenius kernel of $G$, as defined for example in 
\cite[I.9.4]{Jantzen}; in particular, $G_{p^n}$ is a normal 
infinitesimal subgroup scheme of $G$. Then the above 
assertion (i) just means that the union of the iterated 
Frobenius kernels is dense in $G$. 
\end{remark}

We may now prove Theorem \ref{thm:ahv}. Recall its assumptions:
$X$ is a normal projective variety on which a smooth connected 
linear algebraic group $G$ acts with an open dense orbit. 
The variety $X$ is unirational in view of \cite[Thm.~18.2]{Bor91}; 
thus, its Albanese variety is trivial. By duality, it follows that the
Picard variety $\Pic^0(X)$ is trivial as well (see \cite[9.5.25]{FGA}).
Therefore, $\Pic(X) = \NS(X)$ is fixed pointwise by $\Aut^0(X)$. 
Using Lemma \ref{lem:big}, this implies that $\Aut^0(X)$
is linear. We may thus assume that $G = \Aut^0(X)$; 
in particular, $G$ is a normal subgroup of $\Aut(X)$.

Denote by $X_0 \subset X$ the open $G$-orbit; then $X_0$
is normalized by $\Aut(X)$. Choose $x \in X_0(k)$ and 
denote by $H$ its scheme-theoretic stabilizer in $G$. 
Then we have $X_0 = G \cdot x \simeq G/H$ equivariantly 
for the $G$-action. We also have 
$X_0 = \Aut(X) \cdot x \simeq \Aut(X)/\Aut(X,x)$
equivariantly for the $\Aut(X)$-action.
As a consequence, $\Aut(X) = G \cdot \Aut(X,x)$.

Next, choose a positive integer $n$ such that
$N_G(H^0) = N_G(H_n)$ (Lemma \ref{lem:tits}).
The action of $\Aut(X)$ on $G$ by conjugation 
normalizes $G_n$ and induces a linear representation
of $\Aut(X)$ in $V := \cO(G_n)$, a finite-dimensional
vector space. The ideal of $H_n$ is a subspace 
$W \subset V$, and its stabilizer in $G$ equals 
$N_G(H^0)$. We consider $W$ as a $k$-rational point of 
the Grassmannian $\Grass(V)$ parameterizing linear
subspaces of $V$ of the appropriate dimension. The linear
action of $\Aut(X)$ on $V$ yields an action on $\Grass(V)$.
The subgroup scheme $\Aut(X,x)$ fixes $W$, since
it normalizes $\Aut(X,x)^0 = H^0$. Thus, we obtain
\[ \Aut(X) \cdot W = G \cdot \Aut(X,x) \cdot W 
= G \cdot W \simeq G/N_G(H^0). \]
As a consequence, the morphism $f : G/H \to G/N_G(H^0)$
yields an $\Aut(X)$-equivariant morphism
\[ \tau : X_0 \longrightarrow Y, \]
where $Y$ denotes the closure of $\Aut(X) \cdot W$ in 
$\Grass(V)$. 

We may view $\tau$ as a rational map $X \dasharrow Y$. 
Let $X'$ denote the normalization of the graph of 
this rational map, i.e., of the closure of $X_0$ 
embedded diagonally in $X \times Y$. Then $X'$ is 
a normal projective variety equipped with an action of 
$\Aut(X)$ and with an equivariant morphism 
$f : X' \to X$ which restricts to an isomorphism
above the open orbit $X_0$. By Lemma \ref{lem:closed}, 
the image of $\Aut(X)$ in $\Aut(X')$ is closed; thus, 
it suffices to show that $\Aut(X')$ is a linear 
algebraic group. So we may assume that $\tau$ 
extends to a morphism $X \to Y$, that we will still 
denote by $\tau$. 

Next, consider the boundary, 
$\partial X := X \setminus X_0$, that we view as 
a closed reduced subscheme of $X$; it is normalized
by $\Aut(X)$. Thus, the action of $\Aut(X)$ on $X$ 
lifts to an action on the blowing-up of $X$ along
$\partial X$, and on its normalization. Using 
Lemma \ref{lem:closed} again, we may further
assume that $\partial X$ is the support of an
effective Cartier divisor $\Delta$, normalized
by $\Aut(X)$; thus, the line bundle $\cO_X(\Delta)$ 
is $\Aut(X)$-linearized.

We also have an ample, $\Aut(X)$-linearized line 
bundle $M$ on $Y$, the pull-back of $\cO(1)$ under
the Pl\"ucker embedding of $\Grass(V)$. Thus, there 
exist a positive integer $m$ and
a nonzero section $t \in H^0(Y,M^{\otimes m})$
which vanishes identically on the boundary
$\partial Y := Y \setminus G \cdot W$. 
Then $L := \tau^*(M)$ is an $\Aut(X)$-linearized
line bundle on $X$, equipped with a nonzero
section $s := \tau^*(t)$ which vanishes identically
on $\tau^{-1}(\partial Y) \subset \partial X$.
Denote by $D$ (resp.~$E$) the divisor of zeroes
of $s$ (resp.~$t$). Then $D + \Delta$ is an
effective Cartier divisor on $X$, and we have
\[ X \setminus \Supp(D + \Delta) = 
X_0 \setminus \Supp(D) = 
f^{-1}(G \cdot W \setminus \Supp(E)).\]
Since $\partial Y \subset \Supp(E)$, we have
$G \cdot W \setminus \Supp(E) = Y \setminus \Supp(E)$.
The latter is affine as $M$ is ample. Since the morphism 
$f$ is affine (Lemma \ref{lem:tits}), it follows
that $X \setminus \Supp(D + \Delta)$ is affine
as well. Hence $D + \Delta$ is big (Lemma \ref{lem:affbig}).
Also, $\cO_X(D + \Delta) = L\otimes \cO_X(\Delta)$ is
$\Aut(X)$-linearized. In view of Lemma \ref{lem:big}, 
we conclude that $\Aut(X)$ is a linear algebraic group.

\section{Proof of Theorem \ref{thm:inf}}
\label{sec:inf}

By \cite[Thm.~2]{Br10}, there exists a $G$-equivariant morphism
\[ f: X \longrightarrow G/H \]
for some subgroup scheme $H \subset G$ such that 
$H \supset G_{\aff}$ and $H/G_{\aff}$ is finite; equivalently,
$H$ is affine and $G/H$ is an abelian variety. 
Then the natural map $A = G/G_{\aff} \to G/H$
is an isogeny. Denote by $Y$ the scheme-theoretic fiber of $f$ 
at the origin of $G/H$; then $Y$ is normalized by $H$,
and the action map 
\[ G \times Y \longrightarrow X, \quad (g,y) \longmapsto g \cdot y \]
factors through an isomorphism 
\[ G \times^H Y \stackrel{\simeq}{\longrightarrow} X, \] 
where $G \times^H Y$ denotes the quotient of $G \times Y$
by the action of $H$ via 
\[ h \cdot (g,y) := (gh^{-1}, h \cdot y). \]
This is the fiber bundle associated with the faithfully flat 
$H$-torsor $G \to G/H$ and the $H$-scheme $Y$. The above 
isomorphism identifies $f$ with the morphism 
$G \times^H Y \to G/H$ obtained from the projection 
$G \times Y \to G$.

We now obtain a reduction to the case where $G$ is
\emph{anti-affine}, i.e., $\cO(G) = k$. 
Recall that $G$ has a largest anti-affine subgroup scheme
$G_{\ant}$; moreover, $G_{\ant}$ is smooth, connected and
centralizes $G$ (see \cite[III.3.8]{DG}). We have the
Rosenlicht decomposition $G = G_{\ant} \cdot G_{\aff}$;
further, the scheme-theoretic intersection 
$G_{\ant} \cap G_{\aff}$ contains $(G_{\ant})_{\aff}$, and 
the quotient $(G_{\ant} \cap G_{\aff}) /(G_{\ant})_{\aff}$ 
is finite (see \cite[Thm.~3.2.3]{BSU}). As a consequence,
\[ G = G_{\ant} \cdot H \simeq 
(G_{\ant} \times H)/ (G_{\ant} \cap H) \quad \text{and}
\quad G/H \simeq G_{\ant}/(G_{\ant} \cap H). \]
Thus, we obtain an isomorphism of schemes
\begin{equation}\label{eqn:var} 
G_{\ant} \times^{G_{\ant} \cap H} Y 
\stackrel{\simeq}{\longrightarrow} X,
\end{equation} 
and an isomorphism of abstract groups
\begin{equation}\label{eqn:gp}
\Aut^H_{\gp}(G) \stackrel{\simeq}{\longrightarrow} 
\Aut^{G_{\ant}\cap H}_{\gp}(G_{\ant}). 
\end{equation}
 
Next, we construct an action of the subgroup 
$\Aut^H_{\gp}(G) \subset \Aut^{G_{\aff}}_{\gp}(G)$ on $X$.
Let $\gamma \in \Aut^H(G)$. Then $\gamma \times \id$
is an automorphism of $G \times Y$, equivariant under the above
action of $H$. Moreover, the quotient map 
\[ \pi : G \times Y \longrightarrow G \times^H Y \] 
is a faithfully flat $H$-torsor, and hence a categorical quotient. 
It follows that there is a unique automorphism 
$\delta$ of $G \times^H Y = X$ such that the diagram
\[ \xymatrix{
G \times Y \ar[r]^{\gamma \times \id} \ar[d]_{\pi} & 
G \times Y \ar[d]^{\pi} \\ X  \ar[r]^{\delta} &  X  \\
} \]
commutes. Clearly, the assignement 
$\gamma \mapsto \delta$ defines a homomorphism
of abstract groups
\[ \varphi : \Aut^H_{\gp}(G) \longrightarrow \Aut(X). \]
We also have a natural homomorphism
\[ 
\psi : \Aut^H_{\gp}(G) \longrightarrow \Aut_{\gp}(G/H). 
\]
By construction, $f$ is equivariant under the
action of $\Aut^H_{\gp}(G)$ on $X$ via $\varphi$,
and its action on $G/H$ via $\psi$. Moreover,
we have in $\Aut(X)$
\[ \varphi(\gamma) \, g \, \varphi(\gamma)^{-1} 
= \gamma(g) \]
for all $\gamma \in \Aut^H_{\gp}(G)$ and $g \in G$.
In particular, the image of the homomorphism $\varphi$ 
normalizes $G$.

\begin{lemma}\label{lem:psi}
With the above notation, $\psi$ is injective.
Moreover, $\Aut^H_{\gp}(G)$ is a subgroup of finite index 
of $\Aut^{G_{\aff}}_{\gp}(G)$.
\end{lemma}

\begin{proof}
To show both assertions, we may assume that $G$ 
is anti-affine by using the isomorphism (\ref{eqn:gp}).
Then $G$ is commutative and hence its endomorphisms 
(of algebraic group) form a ring, $\End_{\gp}(G)$. 
Let $\gamma \in \Aut^H_{\gp}(G)$
such that $\psi(\gamma) = \id_{G/H}$; then 
$\gamma - \id_G \in \End_{\gp}(G)$ takes $G$ to $H$,
and $H$ to the neutral element.
Thus, $\gamma - \id_G$ factors through a homomorphism
$G/H \to H$; but every such homomorphism is trivial,
since $G/H$ is an abelian variety and $H$ is affine.
So $\gamma - \id_G = 0$, proving the first assertion.

For the second assertion, we may replace $H$
with any larger subgroup scheme $K$ such that 
$K/G_{\aff}$ is finite. By \cite[Thm.~1.1]{Br15}, 
there exists a finite subgroup scheme $F \subset G$ 
such that $H = G_{\aff} \cdot F$. Let $n$ denote 
the order of $F$; then $F$ is contained in the 
$n$-torsion subgroup scheme $G[n]$, and hence
$H \subset G_{\aff} \cdot G[n]$.

We now claim that $G[n]$ is finite for any integer $n > 0$. 
If $\charac(k) = p > 0$, then $G$ is a semi-abelian variety 
(see \cite[ Prop.~5.4.1]{BSU})
and the claim follows readily. If $\charac(k) = 0$, then
$G$ is an extension of a semi-abelian variety by a vector
group $U$ (see \cite[\S 5.2]{BSU}. Since the multiplication 
map $n_U$ is an isomorphism, we have $G[n] \simeq (G/U)[n]$; 
this completes the proof of the claim.

By this claim, we may replace $H$ with the larger subgroup
scheme $G_{\aff} \cdot G[n]$ for some integer $n > 0$. 
Then the restriction map 
\[ \rho : \Aut_{\gp}^{G_{\aff}}(G) \to \Aut_{\gp}(G[n]) \]
has kernel $\Aut^H_{\gp}(G)$. Thus, it suffices to show 
that $\rho$ has a finite image. 

Note that the image of $\rho$ is contained in the image of 
the analogous map $\End_{\gp}(G) \to \End_{\gp}(G[n])$.
Moreover, the latter image is a finitely generated abelian group 
(since so is $\End_{\gp}(G)$ in view of \cite[Lem.~5.1.3]{BSU}) 
and is $n$-torsion (since so is $\End_{\gp}(G[n])$). 
This completes the proof.
\end{proof}

\begin{lemma}\label{lem:varphi}
With the above notation, $\varphi$ is injective.
Moreover, its image is the subgroup of 
$\Aut(X)$ which normalizes $G$ and centralizes $Y$; 
this subgroup intersects $G$ trivially.
\end{lemma}

\begin{proof}
Let $\gamma \in \Aut^H_{\gp}(G)$ such that 
$\varphi(\gamma) = \id_X$. In view of the
equivariance of $f$, it follows that 
$\psi(\gamma) = \id_{G/H}$. Thus, 
$\gamma = \id_G$ by Lemma \ref{lem:psi}.
So $\varphi$ is injective; we will therefore
identify $\Aut^H_{\gp}(G)$ with the image of 
$\varphi$.

As already noticed, this image normalizes
$G$; it also centralizes $Y$ by construction.
Conversely, let $u \in \Aut(X)$ normalizing
$G$ and centralizing $Y$. Since $H$ normalizes $Y$,
the commutator $u h u^{-1} h^{-1}$ centralizes $Y$ 
for any schematic point $h \in H$. Also, 
$u h u^{-1} h^{-1} \in G$. But in view of (\ref{eqn:var}), 
we have $X = G_{\ant} \cdot Y$, where $G_{\ant}$ 
is central in $G$. It follows that $u h u^{-1} h^{-1}$ 
centralizes $X$. Hence $u$ centralizes $H$, 
and acts on $G$ by conjugation via some 
$\gamma \in \Aut^H_{\gp}(G)$. For any schematic points 
$g \in G$, $y \in Y$, we have
$u(g \cdot y) = u g u^{-1} u(y) 
= \gamma(g) u(y) = \gamma(g) y$,
that is, $u = \varphi(\gamma)$. 
  
It remains to show that $\Aut^H_{\gp}(G)$ intersects
$G$ trivially. Let $\gamma \in \Aut^H_{\gp}(G)$
such that $\varphi(\gamma) \in G$. Then
$\gamma$ acts on $G/H$ by a translation,
and fixes the origin. So $\psi(\gamma) = \id_{G/H}$,
and $\gamma = \id_G$ by using Lemma \ref{lem:psi}
again.
\end{proof}

Now Theorem \ref{thm:inf} follows by combining
Lemmas \ref{lem:psi} and \ref{lem:varphi}.

\begin{remark}\label{rem:inf}
(i) With the above notation, $\Aut^H_{\gp}(G)$
is the group of integer points of a linear algebraic
group defined over the field of rational numbers. 
Indeed, we may reduce to the case where $G$ is 
anti-affine, as in the beginning of the proof of 
Lemma \ref{lem:psi}. Then $\Aut^H_{\gp}(G)$ 
is the group of units of the ring 
$\End^H_{\gp}(G) =: R$; moreover, the additive 
group of $R$ is free of finite rank (as follows from 
\cite[Lem.~5.1.3]{BSU}). So the group of units
of the finite-dimensional $\bQ$-algebra
$R_{\bQ} := R \otimes_{\bZ} \bQ$ is a closed 
subgroup of $\GL(R_{\bQ})$ (via the regular
representation), and its group of integer
points relative to the lattice $R \subset R_{\bQ}$
is just $\Aut^H_{\gp}(G)$.

In other terms, $\Aut^H_{\gp}(G)$ is an 
\emph{arithmetic group}; it follows e.g.~that this group
is finitely presented (see \cite{Bor62}).

\smallskip \noindent
(ii) The commensurability class of $\Aut_{\gp}^{G_{\aff}}(G)$ 
is an isogeny invariant. Consider indeed an isogeny 
\[ u : G \longrightarrow G', \] 
i.e., $u$ is a faithfully flat homomorphism
and its kernel $F$ is finite. Then $u$ induces isogenies
$G_{\aff} \to G'_{\aff}$ and $G_{\ant} \to G'_{\ant}$. 
In view of Lemma \ref{lem:psi}, we may thus assume 
that $G$ and $G'$ are anti-affine. We may choose 
a positive integer $n$ such that $F$ is contained 
in the $n$-torsion subgroup scheme $G[n]$; also, recall 
from the proof of Lemma \ref{lem:psi} that $G[n]$ is 
finite. Thus, there exists an isogeny 
\[ v: G' \longrightarrow G\] 
such that $v \circ u = n_G$ (the multiplication map by $n$ 
in $G$). Then we have a natural homomorphism 
\[ u_* : \Aut_{\gp}^{G_{\aff} \cdot F}(G) \longrightarrow 
\Aut_{\gp}^{G'_{\aff}}(G') \]
which lies in a commutative diagram
\[ \xymatrix{
\Aut_{\gp}^{G_{\aff} \cdot G[n]}(G) 
\ar[r]^-{u_*} \ar[d] & 
\Aut_{\gp}^{G'_{\aff} \cdot \Ker(v)}(G') 
\ar[r]^-{v_*} \ar[d] & \Aut_{\gp}^{G_{\aff}}(G) 
\\
\Aut_{\gp}^{G_{\aff} \cdot F}(G) 
 \ar[r]^-{u_*} &  \Aut_{\gp}^{G'_{\aff}}(G'),  \\
} \]
where all arrows are injective, and the images of the vertical 
arrows have finite index (see Lemma \ref{lem:psi} again). 
Moreover, the image of the homomorphism
$v_* \circ u_* = (v \circ u)_* = (n_G)_*$ 
has finite index as well, since this homomorphism is identified
with the inclusion of $\Aut_{\gp}^{G_{\aff} \cdot G[n]}(G)$
in $\Aut_{\gp}^{G_{\aff}}(G)$. This yields our assertion.

\smallskip \noindent
(iii) There are many examples of smooth connected algebraic
groups $G$ such that $\Aut_{\gp}^{G_{\aff}}(G)$ is infinite.
The easiest ones are of the form $A \times G_{\aff}$,
where $A$ is an abelian variety such that $\Aut_{\gp}(A)$
is infinite. To construct further examples, 
let $G$ be any smooth connnected algebraic group,
$\alpha : G \to A$ the quotient homomorphism
by $G_{\aff}$, and $h : A \to B$ a non-zero 
homomorphism to an abelian variety. Then 
$G' := G \times B$ is a smooth connected algebraic group 
as well, and the assignment
$(g,b) \mapsto (g, b + h(\alpha(g)))$ defines
an automorphism of $G'$ of infinite order, which
fixes pointwise $G'_{\aff} = G_{\aff}$. 
\end{remark}

\section{Proof of Theorem \ref{thm:pos}}
\label{sec:pos}

By \cite[Thm.~3]{Br10}, the Albanese morphism of $X$ 
is of the form
\[ f : X \longrightarrow \Alb(X) = G/H, \]
where $H$ is an affine subgroup scheme of $G$ containing 
$G_{\aff}$. Thus, we have as in Section \ref{sec:inf} 
\[ X \simeq G \times^H Y \simeq G_{\ant} \times^K Y, \] 
where $K := G_{\ant} \cap H$ and $Y$ denotes the 
(scheme-theoretic) fiber of $f$ at the origin of $G/H$. 
Then $Y$ is a closed subscheme of $X$, normalized by $H$.

\begin{lemma}\label{lem:alb}

\begin{enumerate}

\item[{\rm (i)}] $G_{\ant}$ is the largest anti-affine subgroup 
of $\Aut(X)$. In particular, $G_{\ant}$ is normal in $\Aut(X)$.

\item[{\rm (ii)}] $f_*(\cO_X) = \cO_{G/H}$.

\item[{\rm (iii)}] If $K$ is smooth, then $Y$ is a normal 
projective variety, almost homogeneous under 
the reduced neutral component $H^0_{\red}$.

\end{enumerate}

\end{lemma}

\begin{proof}
(i) By \cite[Prop.~5.5.3]{BSU}, $\Aut(X)$ has a largest anti-affine 
subgroup $\Aut(X)_{\ant}$, and this subgroup centralizes $G$. 
Thus, $G' := G \cdot \Aut(X)_{\ant}$ is 
a smooth connected subgroup scheme of $\Aut^0(X)$
containing $G$ as a normal subgroup scheme. As a consequence,
$G'$ normalizes the open $G$-orbit in $X$. So $G/H = G'/H'$ 
for some subgroup scheme $H'$ of $G'$; equivalently,
$G' = G \cdot H'$. Also, $H'$ is affine in view of 
\cite[Prop.~2.3.2]{BSU}. Thus, the quotient group
$G'/G \simeq H'/H$ is affine.
But $G'/G = \Aut_{\ant}(X)/(\Aut_{\ant}(X)\cap G)$ is anti-affine.
Hence $G' = G$, and $\Aut_{\ant}(X) = G_{\ant}$.

(ii) Consider the Stein factorization $f = g \circ h$,
where $g : X' \to G/H$ is finite, and $h : X \to X'$ 
satisfies $h_*(\cO_X) = \cO_{X'}$. By Blanchard's lemma 
(see \cite[Prop.~4.2.1]{BSU}), 
the $G$-action on $X$ descends to a unique action on $X'$ 
such that $g$ is equivariant. As a consequence, 
$X'$ is a normal projective variety, almost homogeneous 
under this action, and $h$ is equivariant as well. 
Let $H' \subset G$ denote the scheme-theoretic stabilizer 
of a $k$-rational point of the open $G$-orbit in $X'$. 
Then $H' \subset H$ and the homogeneous space 
$H/H'$ is finite. It follows that $H' \supset G_{\aff}$, i.e., 
$G/H'$ is an abelian variety. Thus, so is $X'$; then 
$X' = X$ and $h = \id$ by the universal property of 
the Albanese variety.

(iii) Since $K$ is smooth, it normalizes the reduced 
subscheme $Y_{\red} \subset Y$. Thus, 
$G_{\ant} \times^K Y_{\red}$
may be viewed as a closed subscheme of 
$G_{\ant} \times^K Y = X$,
with the same $k$-rational points (since
$Y_{\red}(k) = Y(k)$). It follows that
$X = G_{\ant} \times^K Y_{\red}$,
i.e., $Y$ is reduced. 

Next, let $\eta : \tilde{Y} \to Y$
denote the normalization map. Then the action of
$K$ lifts uniquely to an action on
$\tilde{Y}$. Moreover, the resulting morphism
$G_{\ant} \times^K \tilde{Y} \to X$
is finite and birational, hence an isomorphism.
Thus, $\eta$ is an isomorphism as well,
i.e., $Y$ is normal. Since $Y$ is connected
and closed in $X$, it is a projective variety.

It remains to show that $Y$ is almost homogenous
under $H^0_{\red} =: H'$ (a smooth connected
linear algebraic group). Since the homogeneous space 
$H/H'$ is finite, so is the natural map 
$\varphi : G/H' \to G/H$; 
as a consequence, $G/H'$ is an abelian variety
and $\varphi$ is an isogeny. We have a cartesian square
\[ \xymatrix{
X' := G \times^{H'} Y \ar[r] \ar[d] & G/H' \ar[d]^{\varphi} \\ 
X = G \times^H Y \ar[r] &  G/H.  \\
} \]
Thus, $X'$ is a normal projective variety, 
almost homogeneous under $G$. Its open
$G$-orbit intersects $Y$ along an open subvariety,
which is the unique orbit of $H'$.
\end{proof}

We denote by $\Aut(X,Y)$ the normalizer of $Y$ in 
$\Aut(X)$; then $\Aut(X,Y)$ is the stabilizer of 
the origin for the action of $\Aut(X)$ on 
$\Alb(X) = G/H$. 
Likewise, we denote by $\Aut_{\gp}(G_{\ant}, K)$
(resp.~$\Aut(Y,K)$) the normalizer of $K$ in 
$\Aut_{\gp}(G_{\ant})$ (resp.~$\Aut(Y)$).

\begin{lemma}\label{lem:XY}

\begin{enumerate}

\item[{\rm (i)}] There is an exact sequence
\[  \pi_0(K) \longrightarrow
\pi_0 \, \Aut(X) \longrightarrow 
\pi_0 \, \Aut(X,Y) \longrightarrow 0. \]

\item[{\rm (ii)}] We have a closed immersion
\[ \iota: \Aut(X,Y) \longrightarrow 
\Aut_{\gp}(G_{\ant}, K) \times \Aut(Y,K) \]
with image consisting of the pairs $(\gamma,v)$ 
such that $\gamma\vert_K = \Int(v)\vert_K$.

\end{enumerate}

\end{lemma}

\begin{proof}
(i) Since the normal subgroup scheme $G_{\ant}$ 
of $\Aut(X)$ acts transitively on 
$G/H = \Aut(X)/\Aut(X,Y)$, we have 
$\Aut(X) = G_{\ant} \cdot \Aut(X,Y)$. Moreover,
$G \cap \Aut(X,Y) = H$, hence 
$G_{\ant} \cap \Aut(X,Y) = K$. Thus, we obtain
$\Aut^0(X) = G_{\ant} \cdot \Aut(X,Y)^0$ and
\[ \pi_0 \, \Aut(X) = \Aut(X)/\Aut^0(X) 
= \Aut(X,Y)/(G_{\ant} \cdot \Aut(X,Y)^0 \cap \Aut(X,Y)) \]
\[ = \Aut(X,Y)/(G_{\ant}  \cap \Aut(X,Y)) \cdot \Aut(X,Y)^0
= \Aut(X,Y)/K \cdot \Aut(X,Y)^0. \]
This yields readily the desired exact sequence.

(ii) Let $u \in \Aut(X,Y)$, and $v$ its restriction to $Y$. 
Since $u$ normalizes $G_{\ant}$, we have 
$u(g \cdot y) = \Int(u)(g) \cdot u(y) = \Int(u)(g) \cdot v(y)$
for all schematic points $g \in G_{\ant}$ and $y \in Y$.
Moreover, $u$ normalizes $K$, hence
$\Int(u) \in \Aut_{\gp}(G_{\ant},K)$ and $\Int(v) \in \Aut(Y,K)$.
Since $g \cdot y = g h^{-1} \cdot h \cdot y$ for any
schematic point $h \in H$, we obtain 
$v(h \cdot y) = \Int(u)(h) \cdot v(y)$, i.e., 
$\Int(u) = \Int(v)$ on $K$.
Thus, $u$ is uniquely determined by the pair
$(\Int(u),v)$, and this pair satisfies the assertion. 
Conversely, any pair $(\gamma,v)$ satisfying the assertion
yields an automorphism $u$ of $X$ normalizing $Y$, via 
$u(g \cdot y) := \gamma(g) \cdot v(y)$.
\end{proof}

In view of the above lemma, we identify $\Aut(X,Y)$
with its image in $\Aut_{\gp}(G_{\ant},K) \times \Aut(Y,K)$
via $\iota$. Denote by 
$\rho : \Aut(X,Y) \to \Aut_{\gp}(G_{\ant},K)$
the resulting projection.

\begin{lemma}\label{lem:rho}
The above map $\rho$ induces an exact sequence
\[ \pi_0 \, \Aut^K(Y) \longrightarrow \pi_0 \, \Aut(X,Y)
\longrightarrow I \longrightarrow 1, \]
where $I$ denotes the subgroup of $\Aut_{\gp}(G_{\ant},K)$
consisting of those $\gamma$ such that 
$\gamma \vert_K = \Int(v)\vert_K$ for some $v \in \Aut(Y,K)$.
\end{lemma}

\begin{proof}
By Lemma \ref{lem:XY} (ii), we have an exact sequence
\[ 1 \longrightarrow \Aut^K(Y) \longrightarrow \Aut(X,Y)
\stackrel{\rho}{\longrightarrow} I \longrightarrow 1. \]
Moreover, the connected algebraic group $\Aut(X,Y)^0$ 
centralizes $G_{\ant}$ in view of \cite[Lem.~5.1.3]{BSU}; 
equivalently, $\Aut(X,Y)^0 \subset \Ker(\rho)$. This readily 
yields the assertion.
\end{proof}

We now consider the case where $K$ is \emph{smooth}; 
this holds e.g.~if $\charac(k) = 0$. Then $\Aut(Y)$ is a linear 
algebraic group by Theorem \ref{thm:ahv} and Lemma 
\ref{lem:alb} (iii). 
Thus, so is the subgroup scheme $\Aut^K(Y)$, and hence 
$\pi_0 \, \Aut^K(Y)$ is finite. Together with Lemmas
\ref{lem:XY} and \ref{lem:rho}, it follows that $\pi_0 \, \Aut(X)$
is commensurable with $I$.

To analyze the latter group, we consider the homomorphism
\[ \eta : I \longrightarrow \Aut_{\gp}(K), \quad 
\gamma \longmapsto \gamma\vert_K \] 
with kernel $\Aut_{\gp}^K(G_{\ant})$.
Since $K$ is a commutative linear algebraic group, it has
a unique decomposition
\[ K \simeq D \times U, \] 
where $D$ is diagonalizable and $U$ is unipotent. Thus,
we have
\[ \Aut_{\gp}(G_{\ant},K) = 
\Aut_{\gp}(G_{\ant},D) \cap \Aut_{\gp}(G_{\ant},U), \]
\[\Aut_{\gp}(Y,K) = 
\Aut_{\gp}(Y,D) \cap \Aut_{\gp}(Y,U), \]
\[ \Aut_{\gp}(K) \simeq \Aut_{\gp}(D) \times \Aut_{\gp}(U). \]
Under the latter identification, the image of $\eta$ is contained 
in the product of the images of the natural homomorphisms
\[ \eta_D : \Aut(Y,D) \longrightarrow \Aut_{\gp}(D), \quad
\eta_U : \Aut(Y,U) \longrightarrow \Aut_{\gp}(U). \]
The kernel of $\eta_D$ (resp.~$\eta_U$) equals 
$\Aut^D(Y)$ (resp.~$\Aut^U(Y))$; also, the quotient
$\Aut(Y,D)/\Aut^D(Y)$ is finite in view of the rigidity of 
diagonalizable group schemes (see \cite[II.5.5.10]{DG}). 
Thus, the image of $\eta_D$ is finite as well.

As a consequence, $I$ is a subgroup of finite index of
\begin{equation}\label{eqn:J}
J := \{ \gamma \in \Aut_{\gp}(G_{\ant},K) ~\vert~
\gamma \vert_U = \Int(v)\vert_U \text{ for some }
v \in \Aut(Y,K) \}.
\end{equation}
Also, $\pi_0 \, \Aut(X)$ is commensurable with $J$.

If $\charac(k) > 0$, then $G_{\ant}$ is a semi-abelian 
variety (see \cite[Prop.~5.4.1]{BSU}) and hence $U$ is finite. 
Also, $U$ is smooth since so is $K$. Thus, 
$\Aut_{\gp}(U)$ is finite, and hence the image of $\eta$
is finite as well. Therefore, $\pi_0 \, \Aut(X)$ is commensurable 
with $\Aut_{\gp}^K(G)$, and hence with 
$\Aut_{\gp}^{G_{\aff}}(G)$ by Lemma \ref{lem:psi}. 
This completes the proof of Theorem \ref{thm:pos}  
in the case where $\charac(k) > 0$ and $K$ is smooth.

Next, we handle the case where $\charac(k) > 0$ and $K$ is
arbitrary. Consider the $n$th Frobenius kernel 
$I_n := (G_{\ant})_{p^n} \subset G_{\ant}$, where $n$ is 
a positive integer. Then $I_n \cap K$ is the $n$th Frobenius 
kernel of $K$; thus, the image of $K$ in $G_{\ant}/I_n$ 
is smooth for $n \gg 0$ (see \cite[III.3.6.10]{DG}).
Also, $\Aut(X)$ normalizes $I_n$ (since it normalizes
$G_{\ant}$), and hence acts on the quotient
$X/I_n$. The latter is a normal projective variety,
almost homogeneous under $G/I_n$ (as follows
from the results in \cite[\S 2.4]{Br17}). Moreover, 
$(G/I_n)_{\ant} = G_{\ant}/I_n$ and we have 
\[ \Alb(X/I_n) \simeq \Alb(X)/I_n 
\simeq G_{\ant}/I_n  K \simeq 
(G_{\ant}/I_n)/ (K/ I_n \cap K), \]
where $K/I_n \cap K$ is smooth for $n \gg 0$.

We now claim that the homomorphism 
$\Aut(X) \to \Aut(X/I_n)$ is bijective on $k$-rational
points. Indeed, every $v \in \Aut(X/I_n)(k)$
extends to a unique automorphism of the function
field $k(X)$, since this field is a purely inseparable
extension of $k(X/I_n)$. As $X$ is the normalization of 
$X/I_n$ in $k(X)$, this implies the claim. 

It follows from this claim that the induced map 
$\pi_0 \, \Aut(X) \to \pi_0 \, \Aut(X/I_n)$ is an isomorphism.
Likewise, every algebraic group automorphism of 
$G_{\ant}$ induces an automorphism of 
$G_{\ant}/I_n$ and the resulting map
$\Aut_{\gp}(G_{\ant}) \to \Aut_{\gp}(G_{\ant}/I_n)$
is an isomorphism, which restricts to an isomorphism
\[ \Aut_{\gp}^K(G_{\ant}) \stackrel{\simeq}{\longrightarrow}
\Aut_{\gp}^{K/I_n \cap K}(G_{\ant}/I_n). \]
All of this yields a reduction to the case where $K$ is smooth, 
and hence completes the proof of Theorem \ref{thm:pos} 
when $\charac(k) > 0$.

It remains to treat the case where $\charac(k) = 0$. 
Consider the extension of algebraic groups
\[ 0 \longrightarrow D \times U \longrightarrow G_{\ant}
\longrightarrow A \longrightarrow 0, \]
where $A := G_{\ant}/K = G/H = \Alb(X)$ is an
abelian variety. By \cite[\S 5.5]{BSU}, the above extension 
is classified by a pair of injective homomorphisms
\[ X^*(D) \longrightarrow \wA(k), \quad 
U^{\vee} \longrightarrow H^1(A,\cO_A) = \Lie(\wA), \]
where $X^*(D)$ denotes the character group of $D$,
and $\wA$ stands for the dual abelian variety of $A$. 
The images of these homomorphisms yield a finitely 
generated subgroup $\Lambda \subset \wA(k)$ and a subspace
$V \subset \Lie(\wA)$. Moreover, we may identify
$\Aut_{\gp}(G_{\ant},K)$ with the subgroup of $\Aut_{\gp}(\wA)$ 
which stabilizes $\Lambda$ and $V$.
This identifies the group $J$ defined in (\ref{eqn:J}),
with the subgroup of $\Aut_{\gp}(\wA,\Lambda,V)$ 
consisting of those $\gamma$ such that 
$\gamma \vert_V \in \Aut(Y,K) \vert_V$,
where $\Aut(Y,K)$ acts on $V$ via the dual of its
representation in $U$. Note that $\Aut(Y,K) \vert_V$
is an algebraic subgroup of $\GL(V)$. Therefore, the proof of 
Theorem \ref{thm:pos} will be completed by the following
result due to Ga\"el R\'emond:

\begin{lemma}\label{lem:gael}
Assume that $\charac(k) = 0$. Let $A$ be an abelian variety. 
Let $\Lambda$ be a finitely generated subgroup 
of $A(k)$. Let $V$ be a vector subspace of $\Lie(A)$.
Let $G$ be an algebraic subgroup of $\GL(V)$. Let 
\[ \Gamma := \{ \gamma \in \Aut_{\gp}(A,\Lambda,V) 
~\vert~ \gamma\vert_V \in G \}. \]
Then $\Gamma$ is an arithmetic group.
\end{lemma}

\begin{proof}
By the Lefschetz principle, we may assume that $k$ is a subfield of $\bC$.

As $k$ is algebraically closed, there is no difference between the
automorphisms of $A$ and those of its extension to $\bC$. Thus,
we may assume that $k = \bC$.

We denote $W := \Lie(A)$, $L\subset W$ its period lattice, and $L'$ 
the subgroup of $W$ containing $L$ such that $\Lambda=L'/L$;
then $L'$ is a free abelian group of finite rank. Let 
$\Gamma': =\Aut(L) \times \Aut(L') \subset 
G': = \Aut(L \otimes \bQ)\times \Aut(L' \otimes \bQ)$.
If we choose bases of $L$ and $L'$ of 
rank $r$ and $s$ say then this inclusion reads
$\GL_r(\bZ) \times \GL_s(\bZ) \subset \GL_r(\bQ) \times \GL_s(\bQ)$. 
We see $G'$ as the group of $\bQ$-points of the algebraic group 
$\GL_r \times \GL_s$ over $\bQ$. To show that $\Gamma$ is arithmetic, 
it suffices to show that it is isomorphic with the intersection in
$G'$ of $\Gamma'$ with the $\bQ$-points of some algebraic subgroup
$G''$ of $\GL_r \times \GL_s$ defined over $\bQ$.

Now it is enough to ensure that $G''(\bQ)$ is the set of pairs
$(\varphi,\psi) \in G'$ satisfying the following conditions
(where $\varphi_\bR$ stands for the extension
$\varphi\otimes \id_\bR$ of $\varphi$ to $W=L \otimes \bR$) :

\begin{enumerate}

\item[(1)] $\varphi_\bR$ is a $\bC$-linear endomorphism of $W$,

\item[(2)] $\varphi_\bR(L'\otimes\bQ)\subset L'\otimes\bQ$,

\item[(3)] $\varphi_\bR \vert_{L'\otimes\bQ}=\psi$,

\item[(4)] $\varphi_\bR(V)\subset V$,

\item[(5)] $\varphi_\bR \vert_V\in G$.

\end{enumerate}

Indeed, if $(\varphi,\psi)\in\Gamma'$ satisfies these five conditions
then $\varphi\in\Aut(L)$ induces an automorphism of $A$ thanks to (1), 
it stabilizes $\Lambda\otimes\bQ$ because of (2) and then $\Lambda$
itself by (3), since $\psi\in\Aut(L')$. With (4) it stabilizes $V$ and
(5) yields that it lies in $\Gamma$.

We are thus reduced to showing that these five conditions define
an algebraic subgroup of $\GL_r\times\GL_s$ (over $\bQ$). 
But the subset of $\rM_r(\bQ) \times \rM_s(\bQ)$ consisting 
of pairs $(\varphi,\psi)$ satisfying (1), (2) and (3) is 
a sub-$\bQ$-algebra, so its group of invertible elements comes 
indeed from an algebraic subgroup of $\GL_r \times \GL_s$
over $\bQ$.

On the other hand, (4) and (5) clearly define an algebraic
subgroup of $\GL_r \times \GL_s$ over $\bR$. But as we are only interested
in $\bQ$-points, we may replace this algebraic subgroup by the (Zariski)
closure of its intersection with the $\bQ$-points 
$\GL_r(\bQ) \times \GL_s(\bQ)$. This closure is an algebraic subgroup of 
$\GL_r \times \GL_s$ over $\bQ$ and the $\bQ$-points are the same.
\end{proof}

\begin{remark}\label{rem:pos}
Assume that $\charac(k) =0$. Then by the above arguments, 
$\pi_0 \, \Aut(X)$ is commensurable with $\Aut_{\gp}^{G_{\aff}}(G)$
whenever $K$ is \emph{diagonalizable}, e.g., when $G$ is 
a semi-abelian variety. But this fails in general. Consider indeed 
a non-zero abelian variety $A$ and its universal vector extension,
\[ 0 \longrightarrow U \longrightarrow G 
\longrightarrow A \longrightarrow 0. \]
Then $G$ is a smooth connected algebraic group,
and $G_{\aff} = U \simeq H^1(A,\cO_A)^{\vee}$
is a vector group of the same dimension as $A$.
Moreover, $G$ is anti-affine (see \cite[Prop.~5.4.2]{BSU}).
Let $Y := \bP(U \oplus k)$ be the projective completion
of $U$. Then the action of $U$ on itself by translation
extends to an action on $Y$, and $X := G \times^U Y$
is a smooth projective equivariant completion of $G$
by a projective space bundle over $A$. One may
check that $\Aut_{\gp}^U(G)$ is trivial, and
$\Aut(X) \simeq \Aut(A)$; in particular, the group
$\pi_0\,  \Aut(X) \simeq \Aut_{\gp}(A)$ is not necessarily finite.
\end{remark}

\bibliographystyle{amsalpha}

\end{document}